\documentclass[11pt,francais]{smfart}
\usepackage{axodraw}
\usepackage{amsfonts}
\usepackage{amsmath}
\usepackage{amssymb}
\usepackage{color}
\usepackage{graphicx}
\usepackage{xspace}
\input xy
\xyoption{all}
 \font \eightrm=cmr8

 \newcommand{\nc}{\newcommand}

\newtheorem{thm}{Th\'eor\`eme}
\newtheorem{exam}{Exemple}

\newtheorem{lem}[thm]{Lemme}
\newtheorem{prop}[thm]{Proposition}

\newtheorem{rmk}[thm]{Remarque}

\def\diagramme #1{\vskip 4mm \centerline {#1} \vskip 4mm}

\nc{\BA}{{\Bbb A}} \nc{\CC}{{\Bbb C}} \nc{\DD}{{\Bbb D}}
\nc{\EE}{{\Bbb E}} \nc{\FF}{{\Bbb F}} \nc{\GG}{{\Bbb G}}
\nc{\HH}{{\Bbb H}} \nc{\LL}{{\Bbb L}} \nc{\NN}{{\Bbb N}}
\nc{\PP}{{\Bbb P}} \nc{\QQ}{{\Bbb Q}} \nc{\RR}{{\Bbb R}}
\nc{\TT}{{\Bbb T}} \nc{\VV}{{\Bbb V}} \nc{\ZZ}{{\Bbb Z}}
\nc{\Cal}[1]{{\mathcal {#1}}}
\nc{\mop}[1]{\mathop{\hbox {\rm #1} }}
\nc{\smop}[1]{\mathop{\hbox {\eightrm #1} }}
\nc{\mopl}[1]{\mathop{\hbox {\rm #1} }\limits}
\nc{\frakg}{{\frak g}}
\nc{\g}[1]{{\frak {#1}}}
\def \restr#1{\mathstrut_{\textstyle |}\raise-8pt\hbox{$\scriptstyle #1$}}
\def \srestr#1{\mathstrut_{\scriptstyle |}\hbox to
  -1.5pt{}\raise-4pt\hbox{$\scriptscriptstyle #1$}}
\nc{\wt}{\widetilde}
\nc{\wh}{\widehat}
\nc{\un}{\hbox{\bf 1}}
\nc{\redtext}[1]{\textcolor{red}{\tt #1}}
\nc{\bluetext}[1]{\textcolor{blue}{#1}}
\nc{\comment}[1]{[[{\tt {#1}}]] }
\nc{\R}{{\mathbb R}}

\nc\fleche[1]{\mathop{\hbox to #1 mm{\rightarrowfill}}\limits}
\def\semi{\mathrel{\times}\kern -.85pt\joinrel\mathrel{\raise 1.4pt\hbox{${\scriptscriptstyle |}$}}}

\def\racine{\,{\scalebox{0.07}{
\begin{picture}(29,29) (360,-285)
    \SetWidth{6}
    \SetColor{Black}
    \Vertex(375,-271){20}
  \end{picture}
  }}\,}

\def\echela{\,{\scalebox{0.07}{
\begin{picture}(33,116) (353,-443)
    \SetWidth{6}
    \SetColor{Black}
    \Vertex(369,-428){20}
    \Vertex(369,-341){16}
    \Line(369,-341)(369,-419)
  \end{picture}
  }}\,}
  \def\echelb{\,{\scalebox{0.07}{
   \begin{picture}(33,195) (351,-363)
    \SetWidth{6}
    \SetColor{Black}
    \Vertex(369,-262){16}
    \Line(369,-262)(369,-340)
    \Vertex(368,-348){20}
    \Line(369,-184)(369,-262)
    \Vertex(369,-182){16}
  \end{picture}
  }}\,}
  \def\echelc{\,{\scalebox{0.07}{
   \begin{picture}(33,273) (351,-285)
    \SetWidth{6}
    \SetColor{Black}
    \Vertex(369,-184){16}
    \Line(369,-184)(369,-262)
    \Vertex(368,-270){20}
    \Line(369,-106)(369,-184)
    \Vertex(369,-104){16}
    \Line(369,-28)(369,-106)
    \Vertex(369,-26){16}
  \end{picture}
  }}\,}
  \def\arbrey{\,{\scalebox{0.07}{
  \begin{picture}(147,114) (299,-444)
    \SetWidth{6}
    \SetColor{Black}
    \Vertex(368,-429){20}
    \Vertex(313,-344){16}
    \Vertex(434,-348){16}
    \Line(312,-344)(368,-428)
    \Line(433,-348)(372,-428)
  \end{picture}
}}\,}
\def\arbreza{\,{\scalebox{0.07}{
\begin{picture}(147,200) (299,-358)
    \SetWidth{6}
    \SetColor{Black}
    \Vertex(368,-343){20}
    \Vertex(313,-258){16}
    \Vertex(434,-262){16}
    \Line(312,-258)(368,-342)
    \Line(433,-262)(372,-342)
    \Vertex(313,-172){16}
    \Line(311,-179)(311,-252)
  \end{picture}
}}\,}
\def\arbrezb{\,{\scalebox{0.07}{
\begin{picture}(147,119) (299,-439)
    \SetWidth{6}
    \SetColor{Black}
    \Vertex(368,-424){20}
    \Vertex(313,-339){16}
    \Vertex(434,-343){16}
    \Line(312,-339)(368,-423)
    \Line(433,-343)(372,-423)
    \Line(370,-340)(370,-413)
    \Vertex(371,-334){16}
  \end{picture}
  }}\,}
  \def\arbrema{\,{\scalebox{0.07}{
   \begin{picture}(149,215) (296,-345)
    \SetWidth{6}
    \SetColor{Black}
    \Line(370,-244)(370,-317)
    \Vertex(371,-238){16}
    \Line(310,-149)(366,-233)
    \Line(432,-150)(372,-234)
    \Vertex(310,-144){16}
    \Vertex(433,-144){16}
    \Vertex(371,-330){20}
  \end{picture}
  }}\,}
  \def\arbremb{\,{\scalebox{0.07}{
   \begin{picture}(205,388) (240,-172)
    \SetWidth{6}
    \SetColor{Black}
    \Line(370,-71)(370,-144)
    \Vertex(371,-65){16}
    \Line(310,24)(366,-60)
    \Line(432,23)(372,-61)
    \Vertex(310,29){16}
    \Vertex(433,29){16}
    \Vertex(371,-157){20}
    \Line(310,108)(310,35)
    \Vertex(309,118){16}
    \Line(374,207)(314,123)
    \Vertex(369,202){16}
    \Line(250,207)(306,123)
    \Vertex(254,202){16}
  \end{picture}
}}\,}
\def\arbremc{\,{\scalebox{0.07}{
\begin{picture}(149,306) (296,-254)
    \SetWidth{6}
    \SetColor{Black}
    \Line(370,-153)(370,-226)
    \Vertex(371,-147){16}
    \Line(310,-58)(366,-142)
    \Line(432,-59)(372,-143)
    \Vertex(310,-53){16}
    \Vertex(433,-53){16}
    \Vertex(371,-239){20}
    \Line(371,-65)(371,-138)
    \Vertex(371,-52){16}
    \Line(437,40)(377,-44)
    \Line(311,38)(367,-46)
    \Vertex(311,38){16}
    \Vertex(433,37){16}
  \end{picture}
}}\,}
\def\arbremd{\,{\scalebox{0.07}{
\begin{picture}(275,213) (234,-347)
    \SetWidth{6}
    \SetColor{Black}
    \Vertex(371,-332){20}
    \Line(309,-242)(365,-326)
    \Line(439,-238)(379,-322)
    \Vertex(437,-238){16}
    \Vertex(306,-240){16}
    \Line(245,-149)(301,-233)
    \Line(502,-149)(442,-233)
    \Line(422,-160)(436,-235)
    \Line(320,-157)(307,-231)
    \Vertex(321,-148){16}
    \Vertex(497,-154){16}
    \Vertex(248,-154){16}
    \Vertex(420,-148){16}
  \end{picture}
  }}\,}
  \def\arbreme{\,{\scalebox{0.07}{
  \begin{picture}(205,474) (240,-86)
    \SetWidth{6}
    \SetColor{Black}
    \Line(370,15)(370,-58)
    \Vertex(371,21){16}
    \Line(310,110)(366,26)
    \Line(432,109)(372,25)
    \Vertex(310,115){16}
    \Vertex(433,115){16}
    \Vertex(371,-71){20}
    \Line(310,194)(310,121)
    \Vertex(309,204){16}
    \Line(374,293)(314,209)
    \Vertex(369,288){16}
    \Line(250,293)(306,209)
    \Vertex(254,288){16}
    \Line(254,366)(254,293)
    \Vertex(255,374){16}
  \end{picture}
  }}\,}
  \def\arbremf{\,{\scalebox{0.07}{
  \begin{picture}(205,389) (240,-171)
    \SetWidth{6}
    \SetColor{Black}
    \Line(370,-70)(370,-143)
    \Vertex(371,-64){16}
    \Line(310,25)(366,-59)
    \Line(432,24)(372,-60)
    \Vertex(310,30){16}
    \Vertex(433,30){16}
    \Vertex(371,-156){20}
    \Line(310,109)(310,36)
    \Vertex(309,119){16}
    \Line(374,208)(314,124)
    \Vertex(369,203){16}
    \Line(250,208)(306,124)
    \Vertex(254,203){16}
    \Line(309,200)(309,127)
    \Vertex(310,204){16}
  \end{picture}
  }}\,}
  \def\arbremg{\,{\scalebox{0.07}{
  \begin{picture}(244,379) (201,-181)
    \SetWidth{6}
    \SetColor{Black}
    \Line(370,-80)(370,-153)
    \Vertex(371,-74){16}
    \Line(310,15)(366,-69)
    \Line(432,14)(372,-70)
    \Vertex(310,20){16}
    \Vertex(433,20){16}
    \Vertex(371,-166){20}
    \Vertex(264,98){16}
    \Vertex(359,100){16}
    \Vertex(215,183){16}
    \Vertex(302,184){16}
    \Line(359,99)(315,26)
    \Line(299,175)(265,103)
    \Line(214,182)(259,105)
    \Line(267,91)(308,25)
  \end{picture}
  }}\,}
  \def\arbremh{\,{\scalebox{0.07}{
  \begin{picture}(157,403) (296,-157)
    \SetWidth{6}
    \SetColor{Black}
    \Line(370,-56)(370,-129)
    \Vertex(371,-50){16}
    \Line(310,39)(366,-45)
    \Line(432,38)(372,-46)
    \Vertex(310,44){16}
    \Vertex(433,44){16}
    \Vertex(371,-142){20}
    \Line(371,32)(371,-41)
    \Vertex(371,45){16}
    \Line(372,127)(372,54)
    \Vertex(373,137){16}
    \Line(312,228)(368,144)
    \Line(437,226)(377,142)
    \Vertex(441,232){16}
    \Vertex(314,229){16}
  \end{picture}
  }}\,}
   \def\arbremi{\,{\scalebox{0.07}{
  \begin{picture}(205,388) (240,-172)
    \SetWidth{6}
    \SetColor{Black}
    \Line(370,-71)(370,-144)
    \Vertex(371,-65){16}
    \Line(310,24)(366,-60)
    \Line(432,23)(372,-61)
    \Vertex(310,29){16}
    \Vertex(433,29){16}
    \Vertex(371,-157){20}
    \Line(310,108)(310,35)
    \Vertex(309,118){16}
    \Line(374,207)(314,123)
    \Vertex(369,202){16}
    \Line(250,207)(306,123)
    \Vertex(254,202){16}
    \Line(434,109)(434,36)
    \Vertex(433,115){16}
  \end{picture}
  }}\,}
   \def\arbremj{\,{\scalebox{0.07}{
  \begin{picture}(235,388) (240,-172)
    \SetWidth{6}
    \SetColor{Black}
    \Line(370,-71)(370,-144)
    \Vertex(371,-65){16}
    \Line(310,24)(366,-60)
    \Line(432,23)(372,-61)
    \Vertex(310,29){16}
    \Vertex(433,29){16}
    \Vertex(371,-157){20}
    \Line(310,108)(310,35)
    \Vertex(309,118){16}
    \Line(374,207)(314,123)
    \Vertex(369,202){16}
    \Line(250,207)(306,123)
    \Vertex(254,202){16}
    \Vertex(463,-68){16}
    \Line(459,-73)(379,-153)
  \end{picture}
  }}\,}
  \def\arbremk{\,{\scalebox{0.07}{
   \begin{picture}(206,476) (240,-84)
    \SetWidth{6}
    \SetColor{Black}
    \Line(370,17)(370,-56)
    \Vertex(371,23){16}
    \Vertex(371,-69){20}
    \Line(370,105)(370,32)
    \Vertex(370,116){16}
    \Line(432,200)(372,116)
    \Vertex(434,203){16}
    \Line(310,203)(366,119)
    \Vertex(312,207){16}
    \Line(312,286)(312,213)
    \Vertex(312,292){16}
    \Line(256,375)(312,291)
    \Vertex(254,378){16}
    \Line(377,379)(317,295)
    \Vertex(375,376){16}
  \end{picture}
}}\,}
\def\arbreml{\,{\scalebox{0.07}{
\begin{picture}(205,476) (240,-84)
    \SetWidth{6}
    \SetColor{Black}
    \Line(370,17)(370,-56)
    \Vertex(371,23){16}
    \Vertex(371,-69){20}
    \Line(312,286)(312,213)
    \Vertex(312,292){16}
    \Line(256,375)(312,291)
    \Vertex(254,378){16}
    \Line(377,379)(317,295)
    \Vertex(375,376){16}
    \Line(434,111)(374,27)
    \Line(312,107)(368,23)
    \Vertex(433,113){16}
    \Vertex(311,111){16}
    \Line(312,197)(312,124)
    \Vertex(312,202){16}
  \end{picture}
  }}\,}
  \def\arbremm{\,{\scalebox{0.07}{
  \begin{picture}(152,391) (296,-169)
    \SetWidth{6}
    \SetColor{Black}
    \Line(370,-68)(370,-141)
    \Vertex(371,-62){16}
    \Line(310,27)(366,-57)
    \Line(432,26)(372,-58)
    \Vertex(310,32){16}
    \Vertex(433,32){16}
    \Vertex(371,-154){20}
    \Line(371,20)(371,-53)
    \Vertex(371,33){16}
    \Line(438,122)(378,38)
    \Line(314,116)(370,32)
    \Vertex(312,124){16}
    \Vertex(436,124){16}
    \Line(311,198)(311,125)
    \Vertex(311,208){16}
  \end{picture}
}}\,}
\def\arbremn{\,{\scalebox{0.07}{
 \begin{picture}(152,308) (296,-252)
    \SetWidth{6}
    \SetColor{Black}
    \Line(370,-151)(370,-224)
    \Vertex(371,-145){16}
    \Line(310,-56)(366,-140)
    \Line(432,-57)(372,-141)
    \Vertex(310,-51){16}
    \Vertex(433,-51){16}
    \Vertex(371,-237){20}
    \Line(371,-63)(371,-136)
    \Vertex(371,-50){16}
    \Line(438,39)(378,-45)
    \Line(314,33)(370,-51)
    \Vertex(312,41){16}
    \Vertex(436,41){16}
    \Line(372,32)(372,-41)
    \Vertex(372,42){16}
  \end{picture}
  }}\,}
\def\arbremo{\,{\scalebox{0.07}{
 \begin{picture}(205,307) (243,-253)
    \SetWidth{6}
    \SetColor{Black}
    \Line(370,-152)(370,-225)
    \Vertex(371,-146){16}
    \Line(310,-57)(366,-141)
    \Line(432,-58)(372,-142)
    \Vertex(310,-52){16}
    \Vertex(433,-52){16}
    \Vertex(371,-238){20}
    \Line(371,-64)(371,-137)
    \Vertex(371,-51){16}
    \Line(438,38)(378,-46)
    \Vertex(436,40){16}
    \Vertex(257,40){16}
    \Line(258,32)(308,-52)
    \Line(328,34)(370,-43)
    \Vertex(328,40){16}
  \end{picture}
   }}\,}
   \def\arbremp{\,{\scalebox{0.07}{
   \begin{picture}(199,307) (249,-253)
    \SetWidth{6}
    \SetColor{Black}
    \Line(370,-152)(370,-225)
    \Vertex(371,-146){16}
    \Line(432,-58)(372,-142)
    \Vertex(433,-52){16}
    \Vertex(371,-238){20}
    \Line(371,-64)(371,-137)
    \Vertex(371,-51){16}
    \Line(438,38)(378,-46)
    \Vertex(436,40){16}
    \Vertex(328,40){16}
    \Vertex(322,-53){16}
    \Vertex(263,-52){16}
    \Line(325,-60)(366,-141)
    \Line(264,-51)(367,-146)
    \Line(328,34)(369,-46)
  \end{picture}
  }}\,}
   \def\arbremq{\,{\scalebox{0.07}{
  \begin{picture}(159,403) (294,-157)
    \SetWidth{6}
    \SetColor{Black}
    \Line(370,-56)(370,-129)
    \Vertex(371,-50){16}
    \Vertex(371,-142){20}
    \Line(371,32)(371,-41)
    \Vertex(371,45){16}
    \Line(372,127)(372,54)
    \Vertex(373,137){16}
    \Line(312,228)(368,144)
    \Line(437,226)(377,142)
    \Vertex(441,232){16}
    \Vertex(314,229){16}
    \Line(310,133)(366,49)
    \Line(439,135)(379,51)
    \Vertex(440,136){16}
    \Vertex(308,136){16}
  \end{picture}
  }}\,}
 \def\arbremr{\,{\scalebox{0.07}{
 \begin{picture}(154,306) (296,-254)
    \SetWidth{6}
    \SetColor{Black}
    \Line(370,-153)(370,-226)
    \Vertex(371,-147){16}
    \Line(310,-58)(366,-142)
    \Line(432,-59)(372,-143)
    \Vertex(310,-53){16}
    \Vertex(433,-53){16}
    \Vertex(371,-239){20}
    \Line(371,-65)(371,-138)
    \Vertex(371,-52){16}
    \Line(437,40)(377,-44)
    \Line(311,38)(367,-46)
    \Vertex(311,38){16}
    \Vertex(433,37){16}
    \Line(437,-149)(377,-233)
    \Vertex(438,-149){16}
  \end{picture}
  }}\,}
  \def\arbrems{\,{\scalebox{0.07}{
   \begin{picture}(274,303) (237,-257)
    \SetWidth{6}
    \SetColor{Black}
    \Line(370,-156)(370,-229)
    \Vertex(371,-150){16}
    \Line(310,-61)(366,-145)
    \Line(432,-62)(372,-146)
    \Vertex(310,-56){16}
    \Vertex(433,-56){16}
    \Vertex(371,-242){20}
    \Line(254,25)(310,-59)
    \Vertex(251,31){16}
    \Vertex(499,28){16}
    \Line(492,22)(432,-62)
    \Vertex(409,30){16}
    \Vertex(340,32){16}
    \Line(341,29)(310,-50)
    \Line(409,23)(432,-52)
  \end{picture}
  }}\,}
\def\arbremt{\,{\scalebox{0.07}{
\begin{picture}(275,286) (234,-274)
    \SetWidth{6}
    \SetColor{Black}
    \Vertex(371,-259){20}
    \Line(309,-169)(365,-253)
    \Line(439,-165)(379,-249)
    \Vertex(437,-165){16}
    \Vertex(306,-167){16}
    \Line(245,-76)(301,-160)
    \Line(502,-76)(442,-160)
    \Line(422,-87)(436,-162)
    \Vertex(321,-75){16}
    \Vertex(497,-81){16}
    \Vertex(248,-81){16}
    \Vertex(420,-75){16}
    \Line(319,-83)(306,-157)
    \Line(249,-6)(249,-74)
    \Vertex(250,-2){16}
  \end{picture}
}}\,}
   \def\arbremu{\,{\scalebox{0.07}{
  \begin{picture}(275,209) (234,-351)
    \SetWidth{6}
    \SetColor{Black}
    \Vertex(371,-336){20}
    \Line(309,-246)(365,-330)
    \Line(439,-242)(379,-326)
    \Vertex(437,-242){16}
    \Vertex(306,-244){16}
    \Line(245,-153)(301,-237)
    \Line(502,-153)(442,-237)
    \Line(422,-164)(436,-239)
    \Vertex(497,-158){16}
    \Vertex(248,-158){16}
    \Line(304,-157)(304,-237)
    \Vertex(304,-157){16}
    \Vertex(357,-156){16}
    \Line(355,-159)(309,-238)
    \Vertex(421,-157){16}
  \end{picture}
  }}\,}
   \def\arbremv{\,{\scalebox{0.07}{
  \begin{picture}(275,208) (234,-352)
    \SetWidth{6}
    \SetColor{Black}
    \Vertex(371,-337){20}
    \Line(309,-247)(365,-331)
    \Line(439,-243)(379,-327)
    \Vertex(437,-243){16}
    \Vertex(306,-245){16}
    \Line(245,-154)(301,-238)
    \Line(502,-154)(442,-238)
    \Line(422,-165)(436,-240)
    \Vertex(497,-159){16}
    \Vertex(248,-159){16}
    \Vertex(421,-158){16}
    \Line(371,-248)(371,-328)
    \Vertex(371,-241){16}
    \Line(338,-162)(307,-241)
    \Vertex(336,-159){16}
  \end{picture}
  }}\,}
   \def\arbremw{\,{\scalebox{0.07}{
  \begin{picture}(309,284) (200,-276)
    \SetWidth{6}
    \SetColor{Black}
    \Vertex(371,-261){20}
    \Line(309,-171)(365,-255)
    \Line(439,-167)(379,-251)
    \Vertex(437,-167){16}
    \Vertex(306,-169){16}
    \Line(245,-78)(301,-162)
    \Line(502,-78)(442,-162)
    \Line(422,-89)(436,-164)
    \Vertex(497,-83){16}
    \Vertex(248,-83){16}
    \Vertex(420,-77){16}
    \Vertex(214,-7){16}
    \Vertex(286,-6){16}
    \Line(216,-13)(244,-80)
    \Line(284,-12)(253,-79)
  \end{picture}
  }}\,}

  \def\arbreha{\,{\scalebox{0.07}{
   \begin{picture}(149,299) (296,-261)
    \SetWidth{6}
    \SetColor{Black}
    \Vertex(371,-154){16}
    \Line(310,-65)(366,-149)
    \Line(432,-66)(372,-150)
    \Vertex(310,-60){16}
    \Vertex(433,-60){16}
    \Vertex(371,-246){20}
    \Line(310,16)(310,-57)
    \Line(371,-163)(371,-236)
    \Vertex(311,24){16}
  \end{picture}
  }}\,}
  \def\arbrehb{\,{\scalebox{0.07}{
\begin{picture}(149,217) (296,-343)
    \SetWidth{6}
    \SetColor{Black}
    \Vertex(371,-236){16}
    \Line(310,-147)(366,-231)
    \Line(432,-148)(372,-232)
    \Vertex(310,-142){16}
    \Vertex(433,-142){16}
    \Vertex(371,-328){20}
    \Line(371,-245)(371,-318)
    \Vertex(370,-140){16}
    \Line(370,-150)(370,-223)
  \end{picture}
  }}\,}
  \def\arbrehc{\,{\scalebox{0.07}{
  \begin{picture}(154,305) (293,-255)
    \SetWidth{6}
    \SetColor{Black}
    \Vertex(371,-148){16}
    \Vertex(371,-240){20}
    \Line(371,-157)(371,-230)
    \Line(370,-66)(370,-139)
    \Vertex(371,-59){16}
    \Line(435,30)(375,-54)
    \Line(308,34)(364,-50)
    \Vertex(435,33){16}
    \Vertex(307,36){16}
  \end{picture}
  }}\,}
  \def\arbrehd{\,{\scalebox{0.07}{
   \begin{picture}(158,215) (296,-345)
    \SetWidth{6}
    \SetColor{Black}
    \Vertex(371,-238){16}
    \Line(310,-149)(366,-233)
    \Line(432,-150)(372,-234)
    \Vertex(310,-144){16}
    \Vertex(433,-144){16}
    \Vertex(371,-330){20}
    \Line(371,-247)(371,-320)
    \Vertex(442,-247){16}
    \Line(440,-251)(376,-327)
  \end{picture}
  }}\,}
\begin{document}
\title{Lois pr\'e-Lie en interaction}

\author{Dominique Manchon}
\address{Universit\'e Blaise Pascal,
         C.N.R.S.-UMR 6620,
         63177 Aubi\`ere, France}       
         \email{manchon@math.univ-bpclermont.fr}
         \urladdr{http://math.univ-bpclermont.fr/~manchon/}
\author{ Abdellatif Sa\" idi}
\address{Universit\'e de Monastir, Unit\'e de recherche physique Math\'ematique, Facult\'e des sciences de Monastir, Avenue de l'environnement
  5019, Tunisie}

         \email{saidiabdellatif@yahoo.fr}

\date{7 juillet 2009}
\begin{abstract}
Dans l'article \cite{CEM}, Damien Calaque, Kurusch Ebrahimi-Fard et Dominique Manchon ont introduit un nouveau coproduit sur
une alg\`ebre commutative de for\^ets d'arbres enracin\'es $\mathcal{H}$. L'alg\`ebre
de Lie des \'el\'ements primitifs du dual gradu\'e $ \mathcal{H}^{0}$ est munie d'une
structure pr\'e-Lie \`a gauche, not\'ee $\vartriangleright$  qui s'exprime en
termes d'insertion d'un arbre dans un autre.
Dans ce travail nous montrons qu'il y a une relation ``de d\'erivation'' reliant cette structure
 \`a la structure pr\'e-Lie \`a gauche de greffe d'un
arbre sur un autre \cite{ChaLiv}, not\'ee $\rightarrow$, obtenue sur l'alg\`ebre de Lie des \'el\'ements primitifs du
dual gradu\'e $\mathcal{H}_{CK}^0$  de l'alg\`ebre de Hopf de Connes-Kreimer
$\mathcal{H}_{CK}$ \cite{CK1}.
\end{abstract}
\begin{altabstract}
In \cite{CEM}, Damien Calaque, Kurusch Ebrahimi-Fard and Dominique Manchon
define a Hopf algebra $\mathcal{H}$ by introducing a new coproduct on a
commutative algebra of rooted forests. The space of primitive elements of the graded dual
$\mathcal{H}^{0}$ is endowed with a left pre-Lie product denoted by
$\vartriangleright,$ defined in terms of insertion of a tree inside another. In
this work we prove a  ``derivation'' relation between the pre-Lie structure $\vartriangleright$
and the left pre-Lie product $\rightarrow$ \cite{ChaLiv} on the space of
primitive elements of the graded dual $\mathcal{H}^{0}_{CK}$ of the
Connes-Kreimer Hopf algebra $\mathcal{H}_{CK}$\cite{CK1}, defined by grafting.  
\end{altabstract}
\maketitle


\tableofcontents


\section{Introduction}
\label{sect:intro}
Soit $K$ un corps. Une $K$-alg\`ebre magmatique $(E,*)$ est un espace vectoriel $E$ muni
d'une application bilin\'eaire qui \`a chaque bipoint $(x,y)$ associe $x*y$. Pour
tout $x\in E$ on note $L_{x}$ l'application lin\'eaire de $E$ d\'efinie par :
\begin{equation}
L_{x}(y)=x*y,~\text{pour tout}~y\in E.
\end{equation}
On munit $E$ du crochet suivant :
\begin{equation}
[x,y]=x*y-y*x.
\end{equation}
Soit $ \mathcal{L}(E)=\{f:E\longrightarrow E,~f~\text{lin\'eaire}\}.$ On munit
$\mathcal{L}(E)$ du crochet de Lie habituel :
$$[f,g]=f\circ g-g\circ f,~\text{pour tout} f,g\in\mathcal{L}(E).$$
$(E,*)$ est dite alg\`ebre pr\'e-Lie \`a gauche \cite{AgraGam} si :
\begin{equation}
L_{[x,y]}=[L_{x},L_{y}],~\text{pour tout}~x,y\in E,
\end{equation}
i.e. 
\begin{equation}
(x*y-y*x)*z=x*(y*z)-y*(x*z),~\text{pour tout}~~x,y,z\in E.
\end{equation}
Ainsi une alg\`ebre pr\'e-Lie \`a gauche \cite{ChaLiv} est un $K$-espace vectoriel $L$ muni d'une
application bilin\'eaire $*$ v\'erifiant la relation suivante: Pour tout $x,y$ et
$z$ dans $L$ :
\begin{equation}\label{eqprelie}
(x*y)*z-x*(y*z)=(y*x)*z-y*(x*z).
\end{equation}
C'est-\`a-dire si on note par $(x,y,z)$ l'expression $(x*y)*z-x*(y*z)$ la
relation (\ref{eqprelie}) s'\'ecrit :
\begin{equation}
(x,y,z)=(y,x,z).
\end{equation}
De la m\^eme mani\`ere si on a :
\begin{equation}
(x,y,z)=(x,z,y)~~~\text{pour tout}~~ x,y,z\in L,
\end{equation}
on dira que $L$ est une alg\`ebre pr\'e-Lie \`a droite.

En particulier, si $(L,*)$ est une alg\`ebre pr\'e-Lie \`a gauche alors $(L,*^{\smop{op}})$
est une alg\`ebre pr\'e-Lie \`a droite o\`u $*^{\smop{op}}$ est d\'efini par :
\begin{equation}
x*^{\smop{op}}y=y*x.
\end{equation}
Une alg\`ebre pr\'e-Lie $(L,*)$ munie du crochet :
\begin{equation}
[x,y]=x*y-y*x~~ \text{pour tout}~x,y\in L,
\end{equation}
est une alg\`ebre de Lie.
\section{L'alg\`ebre de Hopf de Connes-Kreimer  $\mathcal{H}_{CK}$}
\subsection{D\'efinition}
Soit $\mathcal{T}_n $ l'espace vectoriel engendr\'e par les arbres enracin\'es
\`a $n$ sommets. Soit $\mathcal{T}$ l'espace vectoriel
gradu\'e $\bigoplus_{n\geq 1} \mathcal{T}_{n}$. L'alg\`ebre de Hopf de Connes-Kreimer \cite{CK1} est d\'efinie comme l'alg\`ebre
sym\'etrique de $\mathcal{T}$, i.e :~ $ \mathcal{H}_{CK} =S(\mathcal{T}) $. On note $\un$ l'unit\'e de
$\mathcal{H}_{CK}$ identifi\'ee \`a l'arbre vide. Le coproduit $\Delta_{CK}$ \cite{F} est d\'efini par :
\begin{equation}
\Delta_{CK}(t)=\un\otimes t+t\otimes\un+\sum_{c\in \mathcal{A}dm(t)}{P^{c}(t)\otimes R^{c}(t)},
\end{equation} 
o\`u $\mathcal{A}dm(t)$ d\'esigne l'ensemble des coupes admissibles $c$ de l'arbre
$t$, o\`u $P^{c}(t)$ correspond au branchage et $R^{c}(t)$ correspond au
tronc. On rappelle qu'une coupe est un ensemble non vide d'ar\^etes. Une coupe est dite non admissible s'il existe un sommet de
$t$ tel que le trajet de la 
racine \`a ce dernier rencontre au moins $2$ ar\^etes de la coupe. Une coupe
\'el\'ementaire est une coupe contenant une seule ar\^ete de $t$. Si $ F=t_{1}.t_{2}...t_{n}$ est  une for\^et, on a :
\begin{equation}
\Delta_{CK}(F)=\Delta_{CK}(t_{1})\Delta_{CK}(t_{2})...\Delta_{CK}(t_{n}).
\end{equation}

$\mathcal{H}_{CK}$ est une alg\`ebre de Hopf gradu\'ee, la graduation est
suivant le nombre de sommets.
\begin{exam}
\begin{equation}
\Delta_{CK}(\arbreza)=\un\otimes\arbreza+\arbreza\otimes
\un+\racine\otimes\echelb+\echela\otimes\echela+\racine\otimes\arbrey+\echela\racine\otimes\racine
+\racine\racine\otimes\echela
\end{equation}
\end{exam}
\subsection{Dual gradu\'e de $ \mathcal{H}_{CK}$}
Soit $\mathcal{H}_{CK} ^{0}$ le dual gradu\'e de $\mathcal{H}_{CK}$ :
$\mathcal{H}_{CK} ^{0}=\bigoplus(\mathcal{H}_{CK,n})^{*}$ o\`u
$\mathcal{H}_{CK,n}$ d\'esigne l'espace vectoriel engendr\'e par les
for\^ets \`a $n$ sommets. Alors $\mathcal{H}_{CK} ^{0}$ est une alg\`ebre de
Hopf \cite{F}. On note $(\delta_{u})$ la base duale dans
$(\mathcal{H}_{CK,n})^{*}$ de la base des for\^ets de degr\'e $n$ dans
$\mathcal{H}_{CK,n}$. La correspondance $u\mapsto \delta_u$ d\'efinit un
isomorphisme lin\'eaire de $\Cal H_{CK}$ sur son dual gradu\'e $\Cal H_{CK}^0$.
\begin{lem} (voir \cite{CEM})
Pour tout $t$ arbre dans $\mathcal{T}$ l'\'el\'ement $\delta_{t}$ est un caract\`ere
infinit\'esimal ( i.e. $\delta_{t}(xy)=\delta_{t}(x)e(y)+e(x)\delta_{t}(y)$ pour
tout $x,y$ arbres dans $\mathcal{T}$, o\`u $e$ est l'el\'ement neutre pour le
produit de convolution $\ast$). 
\end{lem}
\begin{lem}
L'espace des  caract\`eres infinit\'esimaux est une alg\`ebre de Lie pour le crochet :
\begin{equation}
[\delta_{t},\delta_{t'}]=\delta_{t}\ast\delta_{t'}-\delta_{t'}\ast\delta_{t}.
\end{equation}
\end{lem}
\begin{proof}
$\delta_{t}$ et $\delta_{t'}$ sont primitifs dans $\mathcal{H}_{CK} ^{0}$ donc
leur crochet de Lie aussi.\\
\end{proof}
\subsection{La loi pr\'e-Lie de greffe}
\label{gr}
 Soient $t,t'$ et $x$ des  arbres dans
$\mathcal{T}$, on d\'efinit :
\begin{equation}
<\delta_{t\rightarrow t'},x>=\sum_{c~\in \mathcal{E}lm(x)}{<\delta_{t},P^{c}(x)><\delta_{t'},R^{c}(x)>},
\end{equation}
o\`u $\mathcal{E}lm(x)$ d\'esigne l'ensemble des coupes  \'el\'ementaires de
$x$. Ainsi
$<\delta_{t\rightarrow t'},x>$ est le nombre de coupes \'el\'ementaires
de $x$ tels que $P^{c}(x)=t$ et $R^{c}(x)=t'.$  Par suite :
\begin{equation}
t\rightarrow t'=\sum_{x~ \text{arbre}}{N(t,t',x)x},
\end{equation}
o\`u $N(t,t',x)$ est le nombre de coupes \'el\'ementaires de $x$ tels que
$P^{c}(x)=t$ et $R^{c}(x)=t'$. On obtient :
\begin{equation}
[\delta_{t},\delta_{t'}]=\delta_{t\rightarrow t' -t'\rightarrow t}.
\end{equation}

\begin{exam}
\begin{equation}
\racine \rightarrow \arbrey = \arbreza + 3 \arbrezb
\end{equation}
 \end{exam}
Pour  un arbre $t$ on note $\sigma(t)$ son facteur de sym\'etrie,
c'est-\`a-dire le nombre d'automorphismes de $t$. On d\'efinit une autre loi $\rightarrow_{\sigma}$ li\'ee \`a la
pr\'ec\'edente comme suit :
\begin{equation}
t \rightarrow_{\sigma} t' =\sum_{x~\text{arbre}}{M(t,t',x)x},
\end{equation}
o\`u $M(t,t',x)=N(t,t',x)\frac{\sigma(t)\sigma(t')}{\sigma(x)}$  est le nombre
de mani\`eres de greffer $t$ sur $t'$ pour obtenir $x$.
\begin{exam}
\begin{equation}
\racine \rightarrow_{\sigma} \arbrey =2 \arbreza + \arbrezb
\end{equation}
\end{exam}
On a de fa\c con \'evidente :
\begin{equation}\label{iso}
\phi(t\to_\sigma t')=\phi(t)\to\phi(t'),
\end{equation}
o\`u $\phi$ est l'isomorphisme lin\'eaire de $\Cal T$ dans $\Cal T$ d\'efini
par $\phi(t)=\sigma(t)t$ pour tout arbre $t$.
\begin{prop}\label{greffe}
Les lois de greffe $\to$ et $\to_\sigma$ sont des lois pr\'e-Lie \`a gauche.
\end{prop}
\begin{proof}
Soient $t,t'$ et $t''$ des arbres dans $\mathcal{T}$, on a :
\begin{equation}
t'\rightarrow_\sigma t'' =\sum_{v\in t''}{t'\circ_{v}t''},
\end{equation}
o\`u $t'\circ_{v}t'':=$ est donn\'e par la greffe de $t'$ sur le sommet $v$
de $t''$. Par suite :\\
\begin{eqnarray}
t\rightarrow_\sigma (t'\rightarrow_\sigma t'')&=&t\rightarrow_\sigma (\sum_{v\in
  t''}{t'\circ_{v}t''})\\
                                        &=&\sum_{v\in t''}{\sum_{w\in
                                            t'}{t\circ_{w}t'\circ_{v}t''}}+\sum_{v\in t''}{\sum_{w\in t''}{t\circ_{w}(t'\circ_{v}t'')}}.
\end{eqnarray}\\
De m\^eme on a :
\begin{equation}
(t\rightarrow_\sigma t')\rightarrow_\sigma t'' =\sum_{v\in t''}{\sum_{w\in
    t'}{t\circ_{w}(t'\circ_{v}t'')}}.
\end{equation}
Ainsi
$ t\rightarrow_\sigma (t'\rightarrow_\sigma t'')-(t\rightarrow_\sigma
t')\rightarrow_\sigma t''=\sum_{v\in t''}{\sum_{w\in
    t''}{t\circ_{w}(t'\circ_{v}t'')}}$. On voit que le membre de droite est
sym\'etrique en $t$ et $t'$, ce qui montre que $\to_\sigma$ est pr\'e-Lie \`a
gauche. L'assertion pour la loi $\to$ s'en d\'eduit imm\'ediatement gr\^ace
\`a \eqref{iso}.\\
\end{proof}
\section{L'alg\`ebre de Hopf $\mathcal{H}$}
\subsection{D\'efinition} \cite{CEM}
Soit $\mathcal{T}'$ l'espace vectoriel engendr\'e par les arbres qui
contiennent au moins une ar\^ete. On consid\`ere $\mathcal{H}$ comme l'alg\`ebre sym\'etrique de $\mathcal{T}'$
i.e. $\mathcal{H}=S(\mathcal{T}')$. L'unit\'e de $\mathcal{H}$ est
identifi\'ee \`a l'arbre sans ar\^etes $\racine$. Cette alg\`ebre est caract\'eris\'ee par
son  coproduit $\Delta$ tel que $\Delta(\racine)=\racine\otimes\racine$ et  pour tout  arbre $t$ diff\'erent de l'arbre vide on d\'efinit :
\begin{equation}
\Delta(t)=\sum_{s \smop{sous-for\^et de}  t}{s\otimes t/s},
\end{equation}
o\`u une sous-for\^et $s$ de $t$ est soit la for\^et triviale $\racine$ et
dans ce cas $t/s =t$, soit
une collection $(t_{1},t_{2},...,t_{n})$ de sous-arbres disjoints de $t$ et
non triviaux (i.e. ayant chacun au moins une ar\^ete). L'arbre $t/s$ est
l'arbre contract\'e que l'on obtient en \'ecrasant chaque composante de la
sous-for\^et sur un point (voir \cite{CEM}).
 \begin{exam}
\begin{equation}
\Delta (\arbreza)=\racine\otimes\arbreza+\arbreza\otimes\racine+2\echela\otimes\arbrey+\echela\otimes\echelb+\echelb\otimes\echela+\arbrey\otimes\echela+\echela\echela\otimes\echela.
\end{equation}
\end{exam}
Cette alg\`ebre est une alg\`ebre de Hopf gradu\'ee \cite{F}, la graduation
\'etant suivant le nombre des ar\^etes.
\subsection{Dual gradu\'e de $\mathcal{H}$}
Soit $\mathcal{H}^{0}$ le dual gradu\'e de $\mathcal{H}$, soit $(Z_s ) ,( s~ \text{for\^et dans }\mathcal{H})$ la base duale du
$\mathcal{H}$. Pour tout $t\in\mathcal{T}'$, $Z_t$ est un caract\`ere infinit\'esimal (donc
primitif dans $\mathcal{H}^{0}$). Si on note par $\star$ le produit de
convolution dans $\mathcal{H}^{0}$, on obtient que l'ensemble des
caract\`eres infinit\'esimaux de $\mathcal{H}$ est une alg\`ebre de Lie
pour le crochet suivant :
\begin{equation}
[Z_{t}, Z_{u}]=Z_{t}\star Z_{u}-Z_{u}\star Z_{t},
\end{equation}
o\`u $t$ et $u$ sont des arbres dans $\mathcal{T}'$. On a :\\
\begin{eqnarray*}
<Z_{t}\star Z_{u},v>&=&<Z_{t}\otimes Z_{u}, \Delta(v)>\\
                    & &\\
                    &=&\sum_{s \smop{dans} v}{<Z_{t},s><Z_{u},v/s >}\\
                    & &\\
                    &=&\mathcal{N}(t,u,v),
\end{eqnarray*}\\
o\`u $\mathcal{N}(t,u,v)$ est le nombre de sous arbres de $v$ isomorphes \`a
$t$ tel que le contract\'e $v/t$ soit isomorphe \`a $u$.
\subsection{La loi pr\'e-Lie d'insertion} 
On d\'efinit la loi $\vartriangleright$ comme l'insertion d'un arbre dans un
autre (voir \cite{CEM}), pour deux arbres $t$ et $t'$ on d\'efinit :
\begin{equation}
t\vartriangleright t' =\sum_{x~\text{arbre}}{\mathcal{N}(t,t',x)x}.
\end{equation} 
On obtient alors :
\begin{equation}
Z_{t}\star Z_{t'}-Z_{t'}\star Z_{t}=Z_{t\vartriangleright
  t'-t'\vartriangleright t}.
\end{equation}
Si $\Pi$ d\'esigne la projection de $\mathcal{H}^{0}$ sur
$Prim(\mathcal{H}^{0})$ parall\`element \`a $\mathcal{H}' =<Z_{f}>$ o\`u $f$ est
une for\^et non triviale (produit d'au moins deux arbres). On note $\widetilde{\Pi} =Id_{\mathcal{H}^{0}}-\Pi$, on aura :
\begin{equation}
Z_{t\vartriangleright t'}=\Pi (Z_{t}\star Z_{t'})
\end{equation}
\begin{lem}
Soit $t$ un arbre et $f$ une for\^et non vide. Alors :
\begin{equation}
\Pi(Z_{t}\star Z_{f})=\Pi(Z_{t}\star\Pi(Z_{f})).
\end{equation} 
\end{lem}
\begin{proof}
-Si $f$ est un arbre on a directement le r\'esultat.\\
-Si $f$ est une for\^et non triviale, soit $u$ un arbre. Alors :\\
\begin{eqnarray*}
<Z_{t}\star Z_{f}, u>&=&<Z_{t}\otimes Z_{f},\Delta(u)>\\
                    &=&\sum_{s~ \text{sous-for\^et de } u}{<Z_{t},s><Z_{f},u/s
                     >}\\
                    &=&0~~\text{car }u/s \text{ arbre.}
\end{eqnarray*}
Ainsi : $$ \Pi(Z_{t}\star Z_{f})=\Pi(Z_{t}\star\Pi(Z_{f}))=0.$$\\
\end{proof}
\begin{lem}\label{pitilde}
Soient $t,u$ deux arbres alors :
\begin{equation}
\widetilde{\Pi}(Z_{t}\star Z_{u})=\widetilde{\Pi}(Z_{u}\star Z_{t}).
\end{equation}
\end{lem}
\begin{proof}
Les \'el\'ements primitifs de $\mathcal{H}^{0}$ forment une alg\`ebre de Lie
pour la convolution $\star$, ce qui prouve :\\
$$\widetilde{\Pi}(Z_{t}\star Z_{u} - Z_{u}\star Z_{t})=0.$$\\
\end{proof}
\begin{prop}
La loi d'insertion est une loi pr\'e-Lie \`a gauche.
\end{prop}
\begin{proof}
Soient $t,u$ et $v$ trois arbres, on a :\\
\begin{eqnarray*}
Z_{t\vartriangleright (u\vartriangleright v)-(t\vartriangleright
  u)\vartriangleright v}&=&\Pi(Z_{t}\star Z_{u \vartriangleright
  v})-\Pi(Z_{t\vartriangleright u }\star Z_{v})\\
                        & &\\
                        &=&\Pi(Z_{t}\star \Pi(Z_{u \vartriangleright v
                        })-\Pi(\Pi(Z_{t}\star Z_{u})\star Z_{v})\\
                        & &\\
                        &=&\Pi(Z_{t}\star Z_{u}\star Z_{v})-\Pi(\Pi(Z_{t}\star
                        Z_{u})\star Z_{v})\\
                        & &\\
                        &=&\Pi (\widetilde{\Pi}(Z_{t}\star Z_{u})\star
                        Z_{v})\\
                        & &\\
                        &=&\Pi(\widetilde{\Pi}(Z_{u}\star Z_{t})\star
                        Z_{v})~~~\mop{(d'apr\`es le lemme \ref{pitilde})}\\
                        & &\\
                        &=&Z_{u\vartriangleright (t\vartriangleright
                          v)-(u\vartriangleright t)\vartriangleright v},
\end{eqnarray*}
ce qui implique que :
$$t\vartriangleright (u\vartriangleright v)-(t\vartriangleright
u)\vartriangleright v =u\vartriangleright (t\vartriangleright
v)-(u\vartriangleright t)\vartriangleright v.$$
Par suite $\vartriangleright$ est une loi pr\'e-Lie \`a gauche.\\
\end{proof}
\begin{rmk}
La proposition \ref{greffe} peut se montrer d'une mani\`ere analogue.
\end{rmk}
\section{Relation entre $\mathcal{H}$ et $\mathcal{H}_{CK}$}
 L'alg\`ebre $\mathcal{H}$ coagit \`a gauche  sur l'alg\`ebre de
 Connes-Kreimer $\mathcal{H}_{CK}$    ( voir \cite{CEM} ) par l'unique
 morphisme d'alg\`ebre $ \Phi
 :\mathcal{H}_{CK}\longrightarrow \mathcal{H}\otimes \mathcal{H}_{CK}$
 d\'efini par : $\Phi (\un)=\racine\otimes \un$ et $\Phi (t)=\Delta(t)$ si
 $t\neq \un $.
\begin{thm}\label{diagramme}(voir le th\'eor\`eme $6$ de \cite{CEM} )
L'application $\Phi$ v\'erifie :
\begin{equation}
(Id_{\mathcal{H}}\otimes\Delta_{CK})\circ \Phi=m^{1,3}\circ
(\Phi\otimes\Phi)\circ\Delta_{CK},
\end{equation}
i.e. le diagramme suivant commute 

\diagramme{
\xymatrix{{\mathcal H}_{CK}\ar[r]^{\Phi}\ar[d]_{\Delta_{CK}}
&{\mathcal H}\otimes{\mathcal H}_{CK}\ar[dd]^{I\otimes\Delta_{CK}}\\
{\mathcal H}_{CK}\otimes{\mathcal H}_{CK}\ar[d]_{\Phi\otimes\Phi}&\\
{\mathcal H}\otimes{\mathcal H}_{CK}\otimes{\mathcal H}\otimes{\mathcal
  H}_{CK}
\ar[r]_{m^{1,3}}&{\mathcal H}\otimes{\mathcal H}_{CK}\otimes{\mathcal H}_{CK}
}
}
o\`u :
\begin{eqnarray*}
m^{1,3}:\mathcal{H}\otimes\mathcal{H}_{CK}\otimes\mathcal{H}\otimes\mathcal{H}_{CK}&\longrightarrow&
\mathcal{H}\otimes\mathcal{H}_{CK}\otimes\mathcal{H}_{CK}\\
 a\otimes b\otimes c\otimes d& \longmapsto &ac\otimes b\otimes d
\end{eqnarray*}
\end{thm}
\begin{prop}\label{neutre}
 Soient $\alpha\in\mathcal{H}^{0}$ et $a\in\mathcal{H}^{0}_{CK}$ alors on a
 :
\begin{equation}
\alpha\star Z_{\racine}=Z_{\racine}\star \alpha=\alpha
\end{equation}
\begin{equation}
Z_{\racine}\star a=a\star Z_{\racine}=a
\end{equation}
\end{prop}
\begin{proof}
voir corollaire $10$ de \cite{CEM}.
\end{proof}
\section{Quatre lois pr\'e-Lie}
En plus des deux lois pr\'e-Lie $\to$ et $\to_\sigma$ d\'efinies au paragraphe
\ref{gr}, on d\'efinit pour $u,v\in \mathcal{T}'$ :
\begin{equation}
u \vartriangleright v =\sum_{w~\text{arbre}}{\mathcal{N}(u,v,w)w},
\end{equation}
o\`u $\mathcal{N}(u,v,w)$ d\'esigne le nombre de sous arbres de $w$ isomorphes
\`a $u$ tels que le contract\'e $w/u$ soit isomorphe \`a $v$.
\begin{exam}
\begin{equation}
\echela \vartriangleright \arbrey = 2 \arbreza +\arbrema +3\arbrezb .
\end{equation}
\end{exam}
On d\'efinit maintenant une autre loi pr\'e-Lie li\'ee \`a la loi
$\vartriangleright$  par $\vartriangleright_{\sigma}$ tels que pour tout
$u,v\in\mathcal{T}'$ :
\begin{equation}
u \vartriangleright_{\sigma} v =\sum_{w~\text{arbre}}{\mathcal{M}(u,v,w)w},
\end{equation}
o\`u
$\mathcal{M}(u,v,w)=\mathcal{N}(u,v,v)\frac{\sigma(u)\sigma(v)}{\sigma(w)}$ repr\'esente le nombre de mani\`eres d'inserer $u$ dans $v$ pour obtenir
$w$.
\begin{exam}
\begin{equation}
\echela \vartriangleright_{\sigma} \arbrey = 4 \arbreza + \arbrema + \arbrezb
\end{equation}
\end{exam}
On a de fa\c con \'evidente une relation analogue \`a \eqref{iso}~:
\begin{equation}\label{iso2}
\phi(u\vartriangleright_{\sigma}v)=\phi(u)\vartriangleright\phi(v),
\end{equation}
La question qui se pose ici est la suivante : y a-t-il une relation reliant les
deux structures pr\'e-Lie $\rightarrow$ et $\vartriangleright$ (ou de fa\c con
\'equivalente, $\rightarrow_\sigma$ et $\vartriangleright_\sigma$)?
\section{Relation entre les lois pr\'e-Lie de greffe et d'insertion}
Dans ce paragraphe nous montrons qu'il y a une relation reliant les deux
structures pr\'e-Lie pr\'ec\'edentes.
\begin{thm}\label{premier theo}
Pour tout $\alpha \in \mathcal{H}^0$ et $a,b \in \mathcal{H}^{0} _{CK}$ on a :
\begin{equation}
\sum_{(\alpha)}{(\alpha_{1}\star a)\ast (\alpha_{2}\star b )}=\alpha
\star (a\ast b),
\end{equation}
o\`u l'on utilise la notation de Sweedler :
  $\Delta(\alpha)=\sum_{(\alpha)}{\alpha_{1} \otimes \alpha_{2}}$
  (\cite{Abe80}, \cite{H1}, \cite{Sw69}).
\end{thm}
\begin{proof}

En dualisant le diagramme du th\'eor\`eme \ref{diagramme} on obtient le r\'esultat.\\
\end{proof}
En cons\'equence si $\alpha\in\mathcal{H}^{0}$ est un caract\`ere
infinit\'esimal, en particulier $\alpha$  est un \'el\'ement primitif donc :
\begin{equation}
\Delta(\alpha)=\alpha\otimes Z_{\racine}+Z_{\racine}\otimes \alpha.
\end{equation}
 Par suite pour tout $a,b\in\mathcal{H}_{CK}^{0}$ on obtient :
\begin{eqnarray*}
\alpha\star (a\ast b)&=&(\alpha\star
a)\ast(Z_{\racine}\star b)+(Z_{\racine}\star a)\ast(\alpha\star b)\\
                                   &=&(\alpha\star a)\ast
                                 b+a\ast(\alpha\star\ b), ~~(\mop{d'apr\`es la
                                   proposition \ref{neutre} }).
\end{eqnarray*}
Soient $\Pi_{CK}$ la projection de $\mathcal{H}_{CK}^{0}$ sur les \'el\'ements
primitifs de $\mathcal{H}_{CK} ^{0}$ parall\'element \`a
$\mathcal{H}_{CK}'=<\delta_f >$ o\`u $f$ est une for\^et non triviale. 
\begin{lem}\label{lemme un}
Soit $f$ une for\^et non vide et $t$ un arbre dans $\mathcal{T}$
alors on a :
\begin{equation}
\Pi_{CK}(Z_{t}\star\delta_{f})=\Pi_{CK}(Z_{t}\star\Pi_{CK}(\delta_{f}))
\end{equation}
\end{lem}
\begin{proof}
- Si $f$ est un arbre, on a directement le r\'esultat car
$\Pi_{CK}(\delta_{f})=\delta_{f}$.\\
- Si $f$ est une for\^et non triviale, soit $u$ un arbre dans
$\mathcal{H}_{CK}$ alors deux cas se pr\'esentent : $u=\un$ ou $u \in
\mop{Ker}\epsilon$, o\`u $\epsilon$ est la co-unit\'e pour l'alg\`ebre de Hopf
$\mathcal{H}_{CK}$.\\
* Si $u=\un$ alors :
\begin{eqnarray*}
< Z_{t}\star \delta_{f},u >&=&< Z_{t}\otimes\delta_{f},\Phi(\un)>\\
                      &=&< Z_{t}\otimes\delta_{f}, \racine \otimes \un >\\
                      &=& 0,
\end{eqnarray*}
or $\Pi_{CK}(\delta_{f})=0,$ donc le r\'esultat est vrai pour l'arbre $\un$.\\
* Si $u\in\mop{Ker}\epsilon$ on a :\\
\begin{eqnarray*}
< Z_{t}\star \delta_{f}, u >&=&< Z_{t}\otimes\delta_{f},\Delta(u) >\\
                           & &\\
                           &=&\sum_{s~~ \smop{sous-for\^et~ de}~
                             u}{Z_{t}(s)\delta_{f}(u/s)}\\
                           & &\\
                           &=& 0,
\end{eqnarray*}
Car le contract\'e $u/s$ est un arbre. Ainsi on obtient que :\\
$$\Pi_{CK}(Z_{t}\star\delta_{f})=0=\Pi_{CK}(Z_{t}\star
\Pi_{CK}(\delta_{f})).$$
Par suite on conclut que :
$$\Pi_{CK}(Z_{t}\star\delta_{f})=\Pi_{CK}(Z_{t}\star\Pi_{CK}(\delta_{f})).$$
\end{proof}
\begin{lem}\label{lemme deux}
Pour tout arbre $t\in\mathcal{T}'$ et $u\in\mathcal{T}$ on a :
\begin{equation}
Z_{t}\star \delta_{u}=\delta_{t\vartriangleright u}
\end{equation}
\end{lem}
\begin{proof}
Soit $w$ un arbre dans $\mathcal{T}$ alors :\\
-Si $w=\un$ alors $<Z_{t}\star \delta_{u},w>=0=<\delta_{t\triangleright u},\un>.$\\
-Si $w\in\mop{Ker}\epsilon$ alors :
\begin{eqnarray*}
<Z_{t}\star \delta_{u},w>&=&<Z_{t}\otimes\delta_{u},\Delta(w)>\\
                         & &\\
                         &=&\sum_{s~ \smop{sous-for\^et de
                           }~w}{<Z_{t},s><\delta_{u},u/s>}\\
                         & &\\
                         &=&<\delta_{t\vartriangleright u},w>.
\end{eqnarray*}
-Si $w=w_{1}w_{2}$ o\`u $w_{1}, w_{2}$ deux arbres .\\
On a $\Phi$ est un morphisme d'alg\`ebre  donc :\\
\begin{eqnarray*}
<Z_{t}\star \delta_{u},w_{1}w_{2}>&=&<Z_{t}\otimes
\delta_{u},\Phi(w_{1})\Phi(w_{2})>\\
                                  & &\\
                                  &=&\sum_{s_{1}~ \smop{de}~ w{_1},~ s_{2}~
                                    \smop{de}~
                                    w_{2}}{<Z_{t},s_{1}s_{2}><\delta_{u},w_{1}/s_{1}w_{2}/s_{2}>}\\
                                  & &\\
                                  &=&0=<\delta_{t\vartriangleright
                                    u},w_{1}w_{2}>.
\end{eqnarray*}
\end{proof}
\begin{thm}\label{joli formule}
Pour tout $t\in \mathcal{T}'$ et $u,v \in \mathcal{T}$ on a :
\begin{equation}
t\vartriangleright(u\rightarrow v)=(t\vartriangleright u)\rightarrow v
+u\rightarrow (t\vartriangleright v).
\end{equation}
\end{thm}
\begin{proof}
Soit $t\in \mathcal{T}'$ et $u,v\in \mathcal{T}$, d'apr\`es cons\'equence du
th\'eor\`eme \ref{premier theo} on a :\\
$$ Z_{t}\star
(\delta_{u}\ast\delta_{v})=(Z_{t}\star\delta_{u})\ast\delta_{v}+\delta_{u}\ast(Z_{t}\star\delta_{v})$$\\
donc :\\
$$\Pi_{CK}(Z_{t}\star(\delta_{u}\ast\delta_{v}))=\Pi_{CK}((Z_{t}\star\delta_{u})\ast\delta_{v})+\Pi_{CK}(\delta_{u}\ast(Z_{t}\star\delta_{v})).$$\\
Par suite :\\
$$\Pi_{CK}(Z_{t}\star\Pi_{CK}(\delta_{u}\ast\delta_{v}))=\Pi_{CK}(\delta_{t
  \vartriangleright u}
\ast\delta_{v})+\Pi_{CK}(\delta_{u}\ast\delta_{t\vartriangleright v}),~
(\text{lemmes \ref{lemme un}+\ref{lemme deux}}).$$\\
Ainsi \\
$$\Pi_{CK}(Z_{t}\star\delta_{u\rightarrow v})=\delta_{(t\vartriangleright
  u)\rightarrow v}+\delta_{u\rightarrow (t\vartriangleright u)},$$\\
d'o\`u :\\
$$\delta_{t\vartriangleright (u\rightarrow v)}=\delta_{(t\vartriangleright
  u)\rightarrow v}+\delta_{u\rightarrow (t\vartriangleright v)},$$\\
ce qui prouve que :\\
$$t\vartriangleright (u\rightarrow v)=(t\vartriangleright u)\rightarrow
v+ u\rightarrow (t\vartriangleright v).$$
\end{proof}
\begin{rmk}
D'apr\`es \eqref{iso}, \eqref{iso2} et le th\'eor\`eme
pr\'ec\'edent on montre aussi que pour tout $t\in\mathcal{T}'$ et $u,v\in \mathcal{T}$ :
\begin{equation}
t\vartriangleright_{\sigma}(u\rightarrow_{\sigma}v)=(t\vartriangleright_{\sigma}u)\rightarrow_{\sigma}v+u\rightarrow_{\sigma}(t\vartriangleright_{\sigma}v).
\end{equation} 
\end{rmk}
\subsection{Quelques calculs explicites}
\begin{exam}
Soient $t=\echela ,u=\racine$ et $v=\echela$ on a :
\begin{equation}
u\rightarrow v= \racine\rightarrow\echela =\echelb + 2 \arbrey ,
\end{equation}
donc :\\
\begin{eqnarray}
t\vartriangleright (u\rightarrow v)&=&\echela\vartriangleright \echelb
+2\echela\vartriangleright \arbrey\\
                                       &=&3 \echelc +2 \arbrema +\arbreza +
                                       (4\arbreza +2 \arbrema +6 \arbrezb )\\
                                       & &\\
                                       &=&3\echelc +4\arbrema +5 \arbreza +
                                       6\arbrezb
\end{eqnarray}\\
Or :
\begin{equation}
t\vartriangleright v =\echela \vartriangleright \echela =2\echelb +2 \arbrey
\end{equation}
donc :\\
$\begin{array}{ccl}
u\rightarrow (t\vartriangleright v)&=&2(\racine\rightarrow\echelb
)+2(\racine\rightarrow \arbrey )\\
                                 &=&2(\echelc +2 \arbrema +\arbreza
                                 )+2(\arbreza +3 \arbrezb )\\
                                 &=& 2\echelc +4\arbrema +4\arbreza +6
                                 \arbrezb
\end{array}$\\
par suite :
\begin{equation}
t\vartriangleright (u\rightarrow v)-u\rightarrow (t\vartriangleright
v)=\echelc +\arbreza ,
\end{equation}
or :
\begin{eqnarray*} 
(t\vartriangleright u)\rightarrow
v&=&(\echela\vartriangleright\racine)\rightarrow\echela\\
 &=&\echela\rightarrow \echela \\
&=&\echelc + \arbreza ,
\end{eqnarray*}
ce qui est la diff\'erence $t\vartriangleright (u\rightarrow
v)-u\rightarrow (t\vartriangleright v)$
\end{exam}
\begin{exam}
On prend dans cet exemple : 
\begin{equation*}
t=\echela,\ 
u=\arbrey,\ 
v=\arbrema.
\end{equation*}\\
- Calcul de $t\vartriangleright (u\rightarrow v )$.\\
On a :\\
\begin{equation*}
u\rightarrow v=\arbrey\rightarrow\arbrema
                  =\arbremb + \arbremc +2 \arbremd .
\end{equation*}\\
Donc :
\begin{equation*}
t\vartriangleright (u\rightarrow
v)=(\echela\vartriangleright\arbremb)+(\echela\vartriangleright\arbremc)+2(\echela\vartriangleright\arbremd).
\end{equation*}
Notons respectivement $(1),(2)$ et $(3)$ les termes pr\'ec\'edents.  On aura
alors :\\
$(1)=2\arbreme+3\arbremf+\arbremg+2\arbremh+2\arbremi+\arbremj+2\arbremk+3\arbreml
,\\
(2)=2\arbremm+3\arbremn+2\arbremo+3\arbremp+2\arbremq+2\arbremh+\arbremr+2\arbrems+\arbremg$\\
et\\
$(3)=4\arbremt+2\arbrems+6\arbremu+2\arbremv+4\arbremw$.\\
Par suite :
\begin{equation}\label{differenceun}
t\vartriangleright (u\rightarrow v)=(1)+(2)+(3).
\end{equation}
-Calcul de $u\rightarrow (t\vartriangleright v)$ :\\
On a :\\
\begin{eqnarray*}
t\vartriangleright v&=&\echela\vartriangleright\arbrema\\
                    &=&2\arbreha+3\arbrehb+2\arbrehc+\arbrehd
\end{eqnarray*}
Par suite :\\
\begin{equation*}
u\rightarrow (t\vartriangleright v)=2(\arbrey\rightarrow\arbreha)
                                       +3(\arbrey\rightarrow\arbrehb)
                                       +2(\arbrey\rightarrow\arbrehc)
                                       +\arbrey\rightarrow\arbrehd ,
\end{equation*}
Notons $(1'),(2'),(3')$ et $(4')$ les termes de la somme pr\'ec\'edente on
aura :\\
$(1')=2(\arbreml+\arbremg+\arbremo+\arbremi+\arbremt),\\
(2')=3(\arbremh+\arbremp+\arbremu),\\
(3')=2(\arbremq+\arbremk+2\arbrems+\arbremw),\\
(4')=\arbremj+\arbremr+2\arbremv+\arbremw .$\\
Ainsi :
\begin{equation}\label{differencedeux}
u\rightarrow (t\vartriangleright v)=(1')+(2')+(3')+(4').
\end{equation}
La diff\'erence (\ref{differenceun})-(\ref{differencedeux}), nous donne :\\
\begin{equation*}
t\vartriangleright (u\rightarrow v)-u\rightarrow (t\vartriangleright
v)=\arbreml +\arbremh+\arbremw+2\arbreme+2\arbremm
   +2\arbremt+3\arbremu+3\arbremf+3\arbremn .
\end{equation*}
-Calcul de $(t\vartriangleright u)\rightarrow v$ :\\
On a :\\
\begin{eqnarray*}
t\vartriangleright u&=&\echela\vartriangleright\arbrey\\
                    &=&\arbrema+2\arbreza+3\arbrezb
\end{eqnarray*}\\
Par suite :\\
\begin{equation*}
(t\vartriangleright u)\rightarrow  v=\arbreml+\arbremh+\arbremw
                                        +2(\arbreme+\arbremm+\arbremt)
                                        +3(\arbremf+\arbremn+\arbremu),
\end{equation*}
ce qui est bien la diff\'erence $t\vartriangleright(u\rightarrow
v)-u\rightarrow (t\vartriangleright v).$
\end{exam}
\subsection{Une approche op\'eradique}\label{operad}
Le th\'eor\`eme \ref{joli formule} est un cas particulier d'un ph\'enom\`ene
g\'en\'eral de nature op\'eradique. Rappelons qu'une op\'erade vectorielle
unitaire $\mathcal{P}$ est la donn\'ee pour tout $n\in\mathbb{N}^*$ d'un espace
vectoriel $\mathcal{P}_{n}$ muni d'une action du groupe sym\'etrique $S_n $,
et d'une loi de composition :
\begin{equation}
\gamma :\mathcal{P}_{n}\otimes\mathcal{P}_{p_1
}\otimes\cdots\otimes\mathcal{P}_{p_n }\longrightarrow \mathcal{P}_{p_1 +p_2
  +\cdots +p_n }
\end{equation}
satisfaisant des axiomes d'associativit\'e et d'\'equivariance par rapport aux
actions des groupes sym\'etriques en pr\'esence. L'unitarit\'e s'\'ecrit
$\mathcal{P}_{1}=K.e$ ( $K$ est le corps de base) o\`u $e$ v\'erifie :
\begin{equation}
\gamma(e;\beta)=\beta
\end{equation}
\begin{equation}
\gamma(\beta;e,e,\ldots,e)=\beta.
\end{equation}  
Pour tout $\alpha\in\mathcal{P}_{n},\beta\in\mathcal{P}_{p}$ et pour tout
$i\in\{1,2,\ldots,n\}$ la composition partielle $\alpha\circ_{i}\beta\in\mathcal{P}_{n+p-1}$ est
d\'efinie par :
\begin{equation}
\alpha\circ_{i}\beta=\gamma(\alpha;\underbrace{e,\ldots,e}_{i-1\smop{termes}},\beta,e,\ldots,e).
\end{equation}
La somme des compositions partielles :
\begin{equation}
\alpha\vartriangleleft\beta:=\sum_{i=1}^{n}{\alpha\circ_{i}\beta}
\end{equation}
est une loi pr\'e-Lie \`a droite \cite{Chap}.\\
L'associativit\'e et l'\'equivariance impliquent notamment :
\begin{enumerate}
\item Pour tout  $ \sigma\in S_n $ et pour tout $\tau\in S_p,$ on a :
$$\sigma\alpha\circ_{i}\tau\beta=(\sigma\circ_{i}\tau)(\alpha\circ_{i}\beta)$$
o\`u $\sigma\circ_{i}\tau$ est la permutation de $\{1,\ldots,n+p-1\}$ obtenue en
permutant $\{i,\ldots,i+p-1\}$ \`a l'aide de $\tau$, puis en permutant
$\{1,\ldots,i-1,\{i+1,\ldots,n+p-1\},i+1,\ldots,n+p-1\}$ \`a l'aide de $\sigma$.\\
\item Pour tout $\mu\in\mathcal{P}_{n}$, pour tout $\beta_{j}\in\mathcal{P}_{p_j
},j=1,\ldots,n$ et pour tout $\alpha\in\mathcal{P}_q$ on a l'\'egalit\'e suivante dans
$\mathcal{P}_{p_1 +\cdots+p_n +q-1 }$ :
\begin{equation}
\gamma(\mu;\beta_1 ,
\ldots,\beta_{j-1},\beta_{j}\circ_{i}\alpha,\beta_{j+1},\ldots,\beta_n
)=\gamma(\mu;\beta_1 ,\ldots,\beta_n )\circ_{p_1 +\cdots+p_{j-1 }+i }\alpha .
\end{equation}
\item Pour tout $\mu\in\mathcal{P}_n ,\beta_j \in\mathcal{P}_{p_j },
j=1,\ldots,n$  et pour tout  $\sigma\in S_n , \tau_{j} \in S_{p_j }$ on a :
$$\gamma(\sigma\mu;\tau_{1} \beta_1 ,\ldots,\tau_{n} \beta_n )=\gamma(\sigma;\tau_{1} ,\ldots,\tau_{n}
)\gamma(\mu;\beta_1 ,\ldots,\beta_n )$$
o\`u $\gamma(\sigma;\tau_{1} ,\ldots,\tau_{n} )$ est la permutation de $\{1,2,\ldots,p_1 +p_2
+\cdots+p_n \}$ obtenue en permutant les \'el\'ements $(p_1
+\cdots+p_{j-1}+1,\ldots,p_1 +\cdots+p_n )$ du j-\`eme bloc par $\tau_{j}$, puis en permutant
les blocs par $\sigma$.
\end{enumerate}
L'alg\`ebre libre \`a un g\'en\'erateur $A_{\mathcal{P}}$ sur l'op\'erade
$\mathcal{P}$ \cite{GLE} est donn\'ee par $\oplus_{n>0}\mathcal{P}_n/S_{n}$. Les
trois \'egalit\'es ci-dessus impliquent que les compositions partielles
passent au quotient ainsi que la loi pr\'e-Lie $\vartriangleleft$, et que nous
avons pour toute op\'eration n-aire $\mu$ sur $A_{\mathcal{P}}$ et $a_j \in
\mathcal{P}_{p_j }/S_{p_j }, j\in\{1,2,\ldots,n\}$
$$\mu(a_1 , \ldots,a_{j_1},a_{j}\circ_{i}b,a_{j+1},\ldots,a_n )=\mu(a_1 ,a_2
,\ldots,a_n )\circ_{p_1 +p_2 +\cdots +p_{j-1}+i }b$$
En sommant sur tous les $i$ possibles de $1$ jusqu'\`a $p_j$ puis sur tout les
$j$ de $1$ \`a $n$, on obtient donc :
\begin{equation}
\sum_{j=1}^{n}{\mu(a_1 ,\ldots,a_{j-1},a_{j}\vartriangleleft b,a_{j+1},\ldots,a_n
  )}=\mu(a_1 ,a_2 ,\ldots,a_n )\vartriangleleft b.
\end{equation}
Le th\'eor\`eme \ref{joli formule} en d\'ecoule en prenant pour $\mathcal{P}$
l'op\'erade pr\'e-Lie \cite{ChaLiv} et pour $\mu$ la loi de greffe \`a droite :
$$\mu(a,b)=a \leftarrow b.$$
On obtient bien :
 $$(a\leftarrow b)\vartriangleleft c=(a\vartriangleleft
c)\leftarrow b+a\leftarrow (b\vartriangleleft c).$$
Le th\'eor\`eme \ref{joli formule} s'en d\'eduit par le fait que :
$a\leftarrow b=b\rightarrow a$ et $a\vartriangleleft b=b\vartriangleright a$.
\subsection*{Remerciements}
Nous remercions vivement Muriel Livernet de nous avoir indiqu\'e l'approche
op\'eradique pr\'esent\'ee au paragraphe \ref{operad}.

\end{document}